\documentclass[10pt, a4paper]{amsart}
\usepackage{amsmath}
\usepackage{amssymb}
\usepackage{amsthm}
\usepackage{graphicx}
\usepackage{url}
\usepackage{enumerate}
\usepackage{xcolor}
\usepackage{mathrsfs} 
\usepackage{t1enc}

\newtheorem{thm}{Theorem}
\newtheorem{lem}[thm]{Lemma}

\newtheorem{cor}[thm]{Corollary}

\theoremstyle{definition}
\newtheorem{defn}[thm]{Definition}
\newtheorem{rem}[thm]{Remark}
\newtheorem{exmp}[thm]{Example}

\DeclareMathOperator{\CR}{CR}
\DeclareMathOperator{\Per}{Per} 
\DeclareMathOperator{\ext}{ext} 
\DeclareMathOperator{\Emp}{\mathfrak{m}} 
\DeclareMathOperator{\supp}{supp} 

\DeclareMathOperator{\conv}{conv}
\DeclareMathOperator{\dist}{dist}

\DeclareMathOperator{\D}{\overrightarrow D}
\DeclareMathOperator{\BD}{BD^*}
\DeclareMathOperator{\di}{\D}
\DeclareMathOperator{\m}{\mathfrak m}

\newcommand{\seq}[2]{\{#1_{#2}\}_{#2=0}^{\infty}}
\newcommand{\emp}[2]{\mathfrak{m}(\underline{#1},#2)}
\newcommand{\empi}[2]{\mathfrak{m}({#1},#2)}
\newcommand{\s}[1]{\underline{#1}=\{#1_n\}_{n=0}^{\infty}}
\newcommand{\und}[1]{\underline{#1}}

\newcommand{\Banach}[1]{\BD(#1)}
\newcommand{\om}[1]{\hat{\omega}({#1})}

\renewcommand{\phi}{\varphi}

\newcommand{\Dirac}[1]{\hat{\delta}({#1})}

\DeclareMathOperator{\M}{\mathcal{M}}
\DeclareMathOperator{\MT}{\mathcal{M}_T}

\DeclareMathOperator{\MTe}{\M_T^e}
\DeclareMathOperator{\MTp}{\M_T^{\text{co}}}

\DeclareMathOperator{\MTpos}{\M_T^+}
\DeclareMathOperator{\MTmix}{\M_T^{\text{mix}}}

\DeclareMathOperator{\Bl}{\mathcal{B}}
\DeclareMathOperator{\lang}{\Bl}

\DeclareMathOperator{\G}{\mathcal{G}}

\DeclareMathOperator{\CsX}{\mathcal{C}^s_X}
\DeclareMathOperator{\CpX}{\mathcal{C}^p_X}

\DeclareMathOperator{\CsY}{\mathcal{C}^s_Y}
\DeclareMathOperator{\CpY}{\mathcal{C}^p_Y}

\DeclareMathOperator{\Orb}{Orb}

\newcommand{\alf}{\mathscr{A}}

\newcommand{\htop}{h_\textrm{top}}
\newcommand{\set}[1]{\left\{#1\right\}}

\newcommand{\eps}{\varepsilon}
\newcommand{\R}{\mathbb{R}}
\newcommand{\Z}{\mathbb{Z}}
\newcommand{\N}{\mathbb{N}}
\newcommand{\Zp}{{\N_0}}

\author{Dominik Kwietniak \and Martha {\L}{\k{a}}cka \and Piotr Oprocha}

\address[D. Kwietniak]{
Faculty of Mathematics and Computer Science, Jagiellonian University in Krakow, ul. \L o\-jasiewicza 6, 30-348 Krak\'ow, Poland}\email{dominik.kwietniak@uj.edu.pl}
\urladdr{www.im.uj.edu.pl/DominikKwietniak/}

\address[M. {\L}{\c{a}}cka]{
Faculty of Mathematics and Computer Science, Jagiellonian University in Krakow, ul. \L o\-jasiewicza 6, 30-348 Krak\'ow, Poland}\email{martha.ubik@uj.edu.pl}

\address[P. Oprocha]{AGH University of Science and Technology, Faculty of Applied
Mathematics, al.
Mickiewicza 30, 30-059 Krak\'ow, Poland}
\email{oprocha@agh.edu.pl}

\title[A panorama of specification-like properties]{A panorama of specification-like properties and their consequences}
\date{\today}

\begin{document}

\begin{abstract} We offer an overview of the specification property, its relatives and their consequences.
We examine relations between specification-like properties and such notions as: mixing, entropy, the structure of the simplex of invariant measures, and various types of the shadowing property. We pay special attention to these connections in the context of symbolic dynamics.
\end{abstract}
\maketitle

The specification property is the ability to find a single point following $\eps$-close an~arbitrary collection of orbit segments, provided that the tracing point is allowed to spend a fixed (dependent on $\eps$) time between consecutive segments.

Rufus~Bowen introduced the specification property in his seminal paper of 1971 on Axiom A diffeomorphisms \cite{Bowen71}. In recent years this notion and its generalizations served as a basis for many developments in the theory of dynamical systems.

This property is closely related to the study of hyperbolic systems initiated during the 1960's. Around that time Stephen Smale noticed that certain maps arising from forced oscillations and geodesic flows on surfaces of negative curvature had similar geometric and analytic properties. This motivated his definition of what we know today as \emph{uniformly hyperbolic systems}. At the same time, the Russian school (an incomplete list contains such names as Anosov, Sinai, Katok) worked intensively on Anosov systems, that is, diffeomorphisms of manifolds under which the whole manifold is hyperbolic. 

Many properties of uniformly hyperbolic systems are consequences of the Specification Theorem \cite[Thm.~18.3.9]{KH}. It states that a diffeomorphism restricted to a~compact locally maximal hyperbolic set has the specification property. This result, together with the closely related Shadowing Theorem \cite[Thm.~18.1.3]{KH} provides tools of great utility in exploring the topological structure and statistical behavior of uniformly hyperbolic systems. There are other important classes of dynamical systems that also have the specification property. Mixing interval maps or, more generally, graph maps, mixing cocyclic shifts (in particular, mixing sofic shifts, and thus shifts of finite type) are among them. Needless to say that this list, although impressive, does not contain all interesting systems. This motivates the search for other properties, call them \emph{specification-like}, which may be used to examine systems without  specification in Bowen's sense.

In this survey we describe various notions designed to replace specification. It turns out that there are many systems lacking the specification property, but exhibiting a weaker version of it, which suffices to derive interesting results. This approach has been used to study systems with some forms of non-uniform hyperbolicity, such as $\beta$-shifts.

\begin{figure}[ht]
\begin{center}
\includegraphics[width=\textwidth]{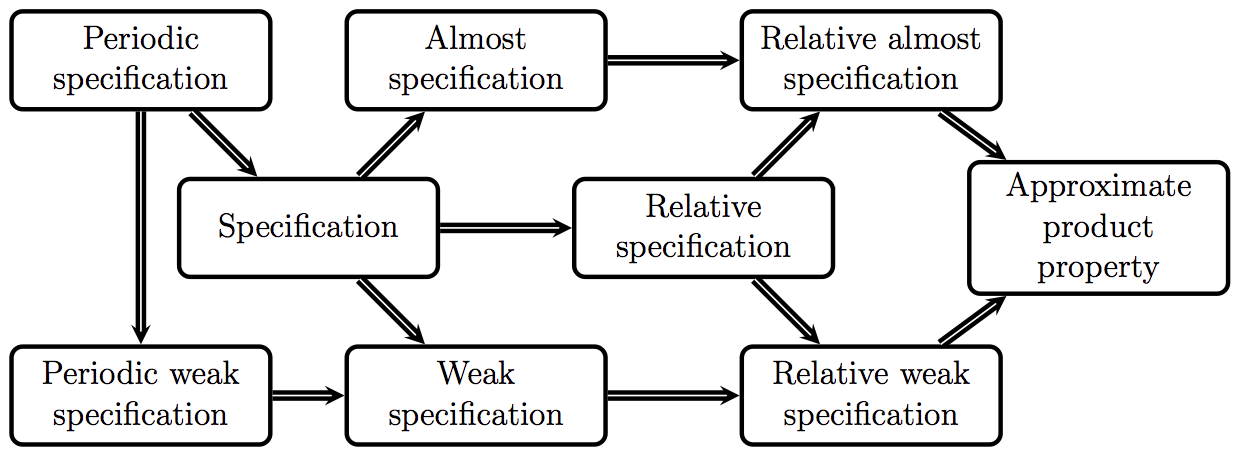}
\end{center}
\caption{The connections between various generalizations of the specification property. There are no more implications between these notions besides those following by transitivity.}
\label{fig:rys}
\end{figure}

The length of this paper does not allow detailed exposition of all aspects of the theory of specification-like properties. We would like to concentrate on the ``big picture'', presenting a broad overview of possible generalizations of the specification property and discussing various examples illustrating dynamical systems with these properties. Figure \ref{fig:rys} presents a diagram summarizing  the specification-like properties we discuss. We also describe examples illustrating the fact that none of the implications presented on Figure \ref{fig:rys} can be reversed. Some of them have never been published before. We would like to  add to this panoramic overview a more detailed (but certainly not complete) account of two problems: intrinsic ergodicity and density of ergodic measures for systems with specification-like properties. Both are related to the structure of the simplex of invariant measures of the dynamical system. 

The first problem is, broadly speaking, a question about the relation between specification-like properties and entropy, both topological and measure-theoretic. One of the first results obtained using specification was that this property together with expansiveness implies the uniqueness of a measure of maximal entropy. Recall that the Variational Principle states that the \emph{topological entropy} $\htop(T)$ of a~compact dynamical system $(X,T)$ equals the supremum over the set of all measure-theoretic entropies $h_\mu(T)$ where $\mu$ runs through all $T$-invariant Borel probability measures on $X$. An invariant measure which achieves this supremum is called the \emph{measure of maximal entropy} for $(X,T)$. A dynamical system $(X,T)$ is \emph{intrinsically ergodic} if it has a unique measure of maximal entropy. We discuss related results in connection with other specification-like properties.

The second problem is an instance of one of the most basic questions in the theory of dynamical systems: \emph{Given a dynamical system, classify and study the properties of invariant measures.} In case that $(X,T)$ has a specification-like property one can usually prove that the ergodic invariant measures are abundant: they form a dense subset of the simplex of invariant measures endowed with the weak$^\ast$-topology.

Among the subjects omitted here are: the role of specification-like properties in the theory of large deviations, specification for flows (actions of reals), and other group actions (for example $\Z^d$ actions with $d\ge 1$). We are also sure that our catalog of specification-like properties is far from being complete. We have selected only those properties, which have clear connections with Bowen's original notion of specification. There are many properties which are fitted only to apply to some very specific examples and their relation with the core of this theory remains unclear. There is also a theory developed by Climenhaga and Thompson and their co-authors which is close in spirit to those notions included here. It certainly deserves attention, but regretfully we have had to leave the comparison of this theory with the specification-like properties presented here to another occasion.

We did try to make this paper accessible to non-specialists, but in some places we had to assume that the reader has some experience with topological dynamics and ergodic theory (as presented, e.g. in \cite{DGS, KH, Walters}). For every result which already exists in the literature the statement itself includes the reference to the original source. But some results we provide are restatements or compilations for which no single reference is appropriate. In these cases we only include the author's name (if such an attribution is possible) in the statement and cite the relevant origins in the preceding paragraph. To make our presentation complete we also introduce a few original results. They mostly come from the second  named author's Master's Thesis written at the Jagiellonian University in Krak\'{o}w under supervision of the first named author. In particular, the results in Section~\ref{sec:app} (on connections between the almost product property and shadowing) 
namely Theorems \ref{thm:shd:aps} and \ref{thm:40}, Corollaries \ref{cor:shad} and \ref{cor:41} and most Examples have not been published before.

\section*{Acknowledgements}
We would like to thank the referee for his/her careful reading and constructive comments.
We are also grateful to Vaughn Climenhaga and Ronnie Pavlov for many discussions on the subject of the paper.

Dominik Kwietniak was supported by the Polish National Science Centre (NCN) grant no. 2013/08/A/ST1/00275; the research of Martha {\L}{\k{a}}cka was supported by the Polish Ministry and Higher Education grant no. {DI2012 002942};  
Piotr Oprocha was supported by the Polish Ministry of Science and Higher Education from sources for science in the years 2013-2014, grant no. IP2012 004272.

The idea of writing this paper together was born in June 2014 during the Activity on Dynamics \& Numbers in the Max Planck Institute of Mathematics in Bonn. The hospitality of MPI is gratefully acknowledged. A preliminary version was derived from the set of notes prepared by the first named author for the mini-course \emph{Specification and shadowing properties} presented by him during \emph{A week on Dynamical Systems} at IM-UFRJ in Rio de Janeiro. Using this occasion he would like to thank the organizers for the invitation and warm reception.

\section{Basic definitions and notation}

\subsection{Notation and some conventions} We write $\N=\{1,2,3,\ldots\}$ and $\Zp=\{0,1,2,\ldots\}$. By $|A|$ we mean the cardinality of a finite set $A$. Given any set $A\subset\Zp$ we write
\begin{itemize}
\item $\overline d(A)$ for the \emph{upper asymptotic density} of $A$, that is,
$$\overline d(A)=\limsup\limits_{n\to\infty}\frac{\big|A\cap\{0,\ldots,n-1\}\big|}{n},$$
\item $\BD(A)$ for the \emph{upper Banach density} of $A$, that is,
$$ \BD(A)=\limsup\limits_{n\to\infty}\max\limits_{k\in\Zp}\frac{\big|A\cap\{k,k+1,\ldots, k+n-1\}\big|}{n}.$$
\end{itemize}
We denote the set of all sequences $\s{x}$ with $x_n$ in some (not necessarily finite) set $A$ for $n=0,1,\ldots$ by $A^\infty$. Recall that a subset of a topological space is of  \emph{first category} if it can be written as countable union of closed nowhere dense sets. It is \emph{residual} if it is a countable intersection of open and dense sets. A set is \emph{nontrivial} if it contains at least two elements.

\subsection{Dynamical systems}
Throughout the paper a \emph{dynamical system} means a~pair $(X,T)$ where $X$ is a~compact metric space and  $T\colon X\to X$ is a 
continuous\footnote{Sometimes we will consider maps that are only piecewise continuous, but in each such case we will indicate this explicitly.} map. We say that $(X,T)$ is \emph{invertible} if $T$ is a homeomorphism. We denote a metric on $X$ by $\rho$. We will often identify a dynamical system $(X,T)$ with a map $T\colon X\to X$ alone.

We say that $x\in X$  is a \emph{periodic point} for $T$ if $T^k(x)=x$ for some $k\in\N$ and we call $k$ a \emph{period} for $x$. 
We denote the set of all periodic points of $T$ by $\Per(T)$.

\subsection{Choquet theory} A nonempty convex compact and metrizable subset $K$ of a~locally convex topological vector space is a \emph{Choquet simplex} if every point of $K$ is the barycenter of a unique probability measure supported on the set of extreme points of $K$ (see \cite{Phelps01}). A \emph{Poulsen simplex} is a nontrivial Choquet simplex $K_P$ such that its extreme points $\ext (K_P)$ are dense in $K_P$. By \cite{LOS} any two nontrivial metrizable Choquet simplices with dense sets of extreme points are equivalent under an affine homeomorphism. Therefore one can speak about \emph{the Poulsen simplex} $K_P$.
It is also known that $\ext (K_P)$ is arcwise connected.

\subsection{Topological dynamics}
We say that $T$ is \emph{transitive} if for every non-empty open sets $U,V\subset X$ there is $n>0$ such that $U\cap T^{-n}(V)\neq\emptyset$.
A dynamical system $(X,T)$ is \emph{(topologically) weakly mixing} when the product system
$(X\times X,T\times T)$ is topologically transitive. A map $T$ is \emph{(topologically) mixing} if for every non-empty open sets $U,V\subset X$ there is $N\in\N$ such that for all $n\ge N$ we have $U\cap T^{-n}(V)\neq\emptyset$. We say that a set $K\subset X$ is \emph{$T$-invariant} if $T(K)\subset K$. A \emph{subsystem} of $(X,T)$ is a pair $(K,T)$, where $K\subset X$ is a nonempty closed $T$-invariant set. Here and elsewhere, we make no distinction between $T$ and its restriction $T|_K$ to a $T$-invariant set $K$ and we often identify a subsystem $(K,T)$ with the set $K$ alone. We say that nonempty closed and $T$-invariant set $K\subset X$ is a~\emph{minimal set} for $(X,T)$ if $(K,T)$ does not contain any proper nonempty subsystem. Given $x\in X$ we define the \emph{orbit} of $x\in X$ as the set $\Orb_T(x)=\set{x,T(x), T^2(x),\ldots}$ and the \emph{orbit closure} of an $x$ as $\overline{\Orb(x,T)}$. A point $x\in X$ is \emph{minimal} if its orbit closure is a minimal set.

\subsection{Invariant measures} Let $\M(X)$ be the set of all Borel probability measures on $X$ equipped with the weak$^\ast$-topology. It is well known that this is a~compact metrizable space (see \cite[\S 6.1]{Walters}). A metric inducing the weak$^\ast$-topology on $\M(X)$ is given by
\[
\D(\mu,\nu)=\inf \{\eps>0\colon \mu(A)\le \nu (A^\eps) +\eps,\text{ for every Borel set }A\subset X\},
\]
where $\mu,\nu\in \M(X)$ and $A^\eps=\{x\in X\colon \rho(x,A)<\eps\}$ denotes the $\eps$-neighborhood of $A$  (see \cite{Strassen65}). The \emph{support} of a measure $\mu\in\M(X)$, denoted by $\supp\mu$, is the smallest closed set $C\subset X$ such that $\mu(C)=1$. We say that $\mu\in \M(X)$ has \emph{full support} if $\supp\mu=X$.

Let $\MT(X)$ denote the set of all $T$-invariant measures in $\M(X)$. By the Krylov-Bogolyubov theorem any dynamical system admits at least one invariant Borel probability measure. We write $\MTe(X)$ for the subset of all ergodic measures. We say that $T$ is \emph{uniquely ergodic} if there is exactly one $T$-invariant measure.

Recall that $\MT(X)$ is a Choquet simplex (see \cite[\S6.2]{Walters}). In particular, $\MT(X)$ is the closure of the convex hull of $\MTe(X)$, thus  $\MTe(X)$ is a nonempty $G_\delta$-subset of $\MT(X)$. Note that $\MT(X)$ is a compact
metric space, hence a subset of $\MT(X)$ is residual if, and only if, it is a dense $G_\delta$.

Recall that $\mu\in\MT(X)$ is \emph{strongly mixing} if for any Borel sets $A,B\subset X$ we have $\mu(A\cap T^{-n}B)\to \mu(A)\mu(B)$ as $n\to\infty$.
We denote by $\MTmix(X)$ the set of all strongly mixing measures.

Let $\MTpos(B)$ denote the set of all $\mu\in\MT(X)$ such that the Borel set $B\subset X$ is a subset of $\supp\mu$. In particular, $\MTpos(X)$ denote the set of all  measures with full support.

We denote by $\MTp(X)$ the set of all invariant measures supported on the~orbit of some periodic point.

\subsection{Generic points}
Let $\Dirac{x}$ denote the point mass measure (Dirac measure) concentrated on $x$.
For any $x\in X$ and $N\in\N$ let $\Emp(x,N)=\frac 1N\sum_{n=0}^{N-1}\hat\delta(T^n(x))$.
A~measure $\mu\in \MT(X)$ is \emph{generated} by $x\in X$ if $\mu$ is a limit of some subsequence of
$\{\Emp(x,n)\}_{n=1}^\infty$. The set of all invariant measures generated by $x\in X$ is denoted by $\om{x}$. We say that $x$ is a \emph{generic point} for $\mu\in\MT(X)$ if $\mu$ is the unique measure generated by $x$.
It is \emph{quasiregular} for $T$ if there exists $\mu\in\MT(X)$ such that $x$ is generic for $\mu$.

\subsection{Measure center}
An open set $U\subset X$ is \emph{universally null} in a dynamical system $(X,T)$ if $\mu(U)=0$ for every $\mu\in\MT(X)$.
The \emph{measure center} of $(X,T)$ is the complement of the union of all universally null sets, or equivalently, it is the smallest closed subset $C$ of $X$ such that $\mu(C)=1$ for every $\mu\in\MT(X)$. Another characterization of the measure center uses ideas of Birkhoff and Hilmy. Birkhoff introduced the \emph{probability of sojourn}, defined for $x\in X$ and $U\subset X$ as
\[
p(x,U)=\limsup_{N\to\infty}\frac{1}{N} \left|\set{0\le n<N:T^n(x)\in U}\right|.
\]
Hilmy \cite{Hilmy36} defined the \emph{minimal center of attraction} of a point $x\in X$ as
$$I(x)=\set{y\in X : p(x,U)>0 \text{ for any neighborhood } U \text{ of }y}.$$
It can be proved (see \cite{Sigmund77}) that the measure center is the smallest closed set containing the minimal center of attraction of every point $x\in X$. If the minimal points are dense in $X$ then the measure center is the whole space, but without density of minimal points no specification-like property we consider can guarantee that.

\subsection{Entropy} Measure-theoretic and topological entropies are among the most important invariants in topological dynamics and ergodic theory. Recall that given a dynamical system $(X,T)$ and an open cover $\mathcal{U}$ of $X$ we define
\[
\mathcal{U}^n=\{U_0\cap T^{-1}(U_1)\cap \ldots \cap T^{-(n-1)}(U_{n-1}):U_j\in \mathcal{U}\text{ for }j=0,1,\ldots,n-1\}.
\]
The \emph{topological entropy} of $(X,T)$ with respect to $\mathcal{U}$ is given by
\[
\htop(T,\mathcal{U})=\lim_{n\to\infty}\frac{1}{n}\log\mathcal{N}(\mathcal{U}^n),
\]
where $\mathcal{N}(\mathcal{U}^n)$ denotes the smallest possible cardinality of an open cover of $X$ formed by elements of $\mathcal{U}^n$.
We will denote by $\htop(T)$ the \emph{topological entropy} of a~map $T$ defined by
\[
\htop(T)=\sup\{\htop(T,\mathcal{U}): \mathcal{U}\text{ is an open cover of }X\}.
\]
For the proof of the existence of the limit above and basic properties of topological entropy see \cite[chapter VII]{Walters}.

For the definition of the \emph{measure-theoretic entropy} $h_{\mu}(T)$ of a~$T$-invariant measure $\mu$ we refer for instance to \cite[chapter IV]{Walters}.



\subsection{Non-wandering set} Given a dynamical system $(X,T)$ let $\Omega(T)$ be the \emph{non-wandering set} of $T$, that is, $x\in X$ belongs to $\Omega(T)$ if for every neighborhood $U$ of $x$ there exists $n>0$ with $T^n(U)\cap U\neq\emptyset$. It is well known that $\Omega(T)$ is a closed invariant subset of $X$.

\subsection{Chain recurrence}
A $\delta$-chain (of length $m$) between $x$ and $y$ is any sequence $\set{x_n}_{n=0}^{m}$ such that $x=x_0$, $ y=x_m$, and $\rho(T(x_n),x_{n+1})<\delta$ for $0\le n <m$.
A point $x$ is \emph{chain recurrent} for $T$ if for every $\delta>0$ there is a $\delta$-chain from $x$ to $x$. The set of all chain recurrent points
is denoted $\CR(T)$. Using compactness, we easily obtain that $\CR(T)$ is a closed set and for every $y\in \CR(T)$ there is $x\in \CR(T)$
such that $T(x)=y$, that is $T(\CR(T))=\CR(T)$.

A dynamical system $(X,T)$ is \emph{chain recurrent} if $X=\CR(T)$.
If for every $x,y\in X$ and every $\delta>0$ there exists a $\delta$-chain from $x$ to $y$ then $(X,T)$ is \emph{chain transitive}.

\subsection{Orbit segments and Bowen balls}
Let $a,b\in\Zp$, $a\leq b$. The \emph{orbit segment} of $x\in X$ over $[a,b]$ is the~sequence
\[
T^{[a,b]}(x)=(T^a(x),T^{a+1}(x),\ldots, T^b(x)).
\]
We also write $T^{[a,b)}(x)=T^{[a,b-1]}(x)$.
A \emph{specification} is a family of orbit segments
\[\xi=\{T^{[a_j,b_j]}(x_j)\}_{j=1}^n\] such that  $n\in\N$ and $b_j<a_{j+1}$ for all $1\le j <n$.
The number of orbit segments in a specification is its \emph{rank}.

The \emph{Bowen distance between $x,y\in X$ along a finite set $\Lambda\subset\Zp$} is
\[
\rho^T_\Lambda(x, y) = \max\{\rho(T^j(x), T^j(y)) : j \in\Lambda\}.
\]
By the \emph{Bowen ball (of radius $\eps$, centered at $x\in X$) along $\Lambda$}
we mean the set
\[
B_\Lambda(x, \eps) = \{y \in X : \rho^T_\Lambda(x, y) < \eps\}.
\]
If $\Lambda=\set{0,1,\ldots,n-1}$ then we simply write $B_n(x, \eps):=B_\Lambda(x, \eps)$ and $\rho^T_n(x, y)=\rho^T_\Lambda(x, y)$.

\subsection{Natural extension}
The \textit{inverse limit space} of  a surjective dynamical system is the space  
\[X_T=\{(x_1,x_2,x_3,\ldots)\in X^\infty:T(x_{i+1})=x_i \text { for all }i\in\N\}.\]
We equip $X_T$ with the subspace topology induced by the product topology on $X^\infty$.
The map $T$ is called a \textit{bonding map}. The map $\sigma_T\colon X_T\to X_T$ given by
$$\sigma_T(x_0,x_1,x_2,\ldots)=(T(x_0),T(x_1),T(x_2),\ldots)=(T(x_0),x_0,x_1,\ldots).$$
is called  the \textit{shift homeomorphism} and the invertible dynamical system $(X_T,\sigma_T)$ is a \emph{natural extension} of $(X,T)$.
Note that if $(X,T)$ is invertible then $(X,T)$ and $(X_T,\sigma_T)$ are conjugate. If $T$ is not invertible, then $(X,T)$ is only a factor of $(X_T,\sigma_T)$.

Dynamical systems $(X,T)$ and $(X_T,\sigma_T)$ share many dynamical properties. For example, it is not hard to check
that one of them is transitive, mixing or has a~specification(-like) property if and only if the other has the respective property.
It was proved in \cite{ChenLi} that the same equivalence holds for the shadowing property.
Furthermore,  the invariant measures of $(X,T)$ and $(X_T,\sigma_T)$ can be identified by a natural entropy preserving bijection.
Hence, $\htop(T)=\htop(\sigma_T)$ (see \cite{Ye95} for a~more general statement).

\subsection{Expansiveness}
An (invertible) dynamical system $(X,T)$ is
\emph{positively expansive (expansive)} if there is a constant $c>0$
such that if $x,y\in X$ satisfy
$d(T^n(x),T^n(y))<c$ for all $n\in \Zp$ (all $n\in \Z$), then $x=y$.

Two-sided shift spaces and Axiom A diffeomorphisms are expansive (see \cite{DGS}).
One-sided shift spaces are positively expansive.
If $(X,T)$ is invertible and positively expansive, then $X$ is a finite set (see \cite{Richeson}). 

If a dynamical system is expansive or positively expansive, then its natural extension is expansive, but the converse is not true (see \cite{AH}, Theorem 2.2.32(3)).


\section{Specification property}

The periodic specification property was introduced by Bowen \cite{Bowen71} as a consequence of topological mixing of an axiom A diffeomorphism. Roughly speaking, the specification property allows to approximate segments of orbits by a single orbit, provided that these segments are sparse enough in time.
Recall that a diffeomorphism $T\colon M\to M$ of a smooth compact manifold satisfies \emph{Smale's Axiom A} if the periodic points of $T$ are dense in the non-wandering set $\Omega(T)$ and the tangent bundle of $M$ restricted to $\Omega(T)$, denoted $T\Omega(M)$, has a continuous splitting $T\Omega(M)=E_s\oplus E_u$ into subspaces invariant under the derivative $DT$ such that the restrictions $DT|E_s$ and $DT^{-1}|E_u$ are contractions. Smale \cite[Theorem~6.2]{Smale67} proved that the non-wandering set of an Axiom A diffeomorphism $T$ is the disjoint union of finitely many \emph{basic sets} which are closed, invariant, and the restriction of $T$ to each of them is topologically transitive.  Furthermore, Bowen proved that if $\Lambda$ is a basic set for $T$, then $\Lambda$ can be decomposed into disjoint closed sets $\Lambda_1,\ldots,\Lambda_m$ such that $T(\Lambda_i)=\Lambda_{(i+1)\bmod m}$ and $T^m|\Lambda_i$ has the periodic specification property and that is how this property entered into mathematics. Some authors call $\Lambda_1,\ldots,\Lambda_m$ \emph{elementary sets}.

\begin{defn}Let $\nu\colon \N\to\N$ be any function.
A~family of orbit segments $\xi=\{T^{[a_j,b_j]}(x_j)\}_{j=1}^n$ is a \emph{$\nu$-spaced specification} if
$a_i-b_{i-1} \ge  \nu(b_i-a_i+1) $ for $2\le i \le n$. Given a \emph{constant} $N\in\N$ by an  \emph{$N$-spaced specification} we mean a $\nu$-spaced specification where $\nu$ is the \emph{constant function} $\nu(n)= N$ for all $n\in\N$.
\end{defn}
\begin{defn}
We say that a specification $\xi=\{T^{[a_j,b_j]}(x_j)\}_{j=1}^n$
is $\eps$-traced by $y\in X$ if
\[
\rho(T^k(y),T^k(x_i)) \le \eps \quad\text{for } a_i\le k \le b_i\text{ and }1\le i \le n.
\]
\end{defn}
\begin{defn}\label{def:spec}
We say that $(X,T)$ has the \emph{specification property} if
for any $\eps>0$ there is a constant $N = N(\eps)\in\N$  such that any
$N$-spaced specification $\xi=\{T^{[a_j,b_j]}(x_j)\}_{j=1}^n$ is $\eps$-traced by some $y\in X$.
If additionally, $y$ can be chosen in such a way that $T^{b_n-a_0+N}(y)=y$
then $(X,T)$ has the \emph{periodic specification property}.
\end{defn}

Some authors consider a weaker notion, which we propose to name the \emph{(periodic) specification property of order $k$}. A dynamical system has (periodic) specification of order $k\in\N$ if for every $\eps>0$ there is an $N$ such that every specification of rank $k$ is $\eps$-traced by some (periodic) point.
This weaker version of the (periodic) specification property may replace the stronger one in many proofs, but we do not know of any  examples showing that these notions differ. We expect that even if they do, the examples demonstrating this would not be ``natural'', that is, these potential examples would be systems defined for the sole purpose of proving that a specification property of finite order does not imply the specification property. Note that for shift spaces the periodic specification property, the specification property, and the specification property of order $k$, where $k\ge 2$ are equivalent (see Lemma \ref{lem:psp} and Section \ref{sec:symbolic}).

It is not hard to see that every map with the periodic specification property is onto, but this is not the case if the map has only specification.
\begin{exmp}
Let $X=\set{0,1}$ and $T\colon X\ni x\mapsto 0\in X$. Then $(X,T)$ has the specification property, but $T$ is not onto.
\end{exmp}
Every map on a one point space has the periodic specification property. Therefore we henceforth concentrate on dynamical systems $(X,T)$ given by an \textbf{onto map on a nontrivial space}. Note that some authors (see for example \cite{Yamamoto09}) use a slightly different definition of the specification property which implies surjectivity and for onto maps is equivalent to Definition \ref{def:spec}. For the sake of completeness we recall that a dynamical system $(X,T)$ has the specification property as defined in \cite{Yamamoto09} if for any $\eps>0$ there is an integer $M_\eps$ such that for any $k\ge 1$ and $k$ points $x_1,\ldots,x_k\in X$ and for any sequence of integers $0\le a_1\le b_1< a_2\le b_2<\ldots<a_k\le b_k$ with $a_{i}-b_{i-1}\ge M_\eps$ for  $2\le i\le k$, there is an $x\in X$ with $\rho(T^{a_i+j}(x),T^j(x_i))\le \eps$ for $0\le j\le b_i-a_i$ and $0\le i\le k$. Note that periodic specification is called \emph{strong specification} in \cite{Yamamoto09}.

Observe that a dynamical system $(X,T)$ is topologically transitive if and only if for every $x_0,\ldots, x_k\in X$ and $n_0,\ldots, n_k\in\mathbb N_0$ there exist $m_1,\ldots,m_k\in\N$ such that
\[
\bigcap_{j=0}^k T^{-\sum_{i=1}^{j}(n_{i-1}+m_i)}B_{n_j}(x_j,\eps)\neq\emptyset.
\]
Clearly, each $m_j$ depends on all points $x_i$, all $n_i$ and $\eps$. Therefore the specification property can be considered as a uniform version of transitivity, which allows us to pick all $m_j$ equal to a constant depending only on $\eps$.

The following theorem summarizes easy consequences of the (periodic) specification property.
\begin{thm}[cf. \cite{DGS}, Propositions 21.3--4]
\begin{enumerate}
\item If $(X,T)$ has the (periodic) specification property then $(X,T^k)$ has the (periodic) specification property for every $k\geq 1$.
\item If $(X,T)$, $(Y,S)$ have the (periodic) specification property then the product system $(X\times Y, T\times S)$ also has the (periodic) specification property.
\item Every factor of a system with the (periodic) specification property has the (periodic) specification property.
\item Every onto map with the specification property is topologically mixing.
\end{enumerate}
\end{thm}

The following fact is a simple consequence of the definition of expansiveness, but due to its importance we single it out as a separate lemma. It is proved implicitly by many authors, and an explicit statement and proof can be found as a part of Lemma 9 in \cite{KO12}.

\begin{lem}[cf. \cite{KO12}, Lemma 9]\label{lem:psp}
If $(X,T)$ has the specification property and its natural extension is expansive, then $(X,T)$ has the periodic specification property.
\end{lem}

Bowen \cite[Proposition 4.3]{Bowen71} proved that any system with the periodic specification property\footnote{Strictly speaking, Bowen assumed that the system is \emph{$C$-dense} (a notion which we do not use in this paper), but his proof applies to systems with the specification property which follows from the $C$-density assumption.} on a nontrivial space has positive topological entropy with respect to any open cover of $X$ by two nondense open sets.
In ergodic theory there is a class of $K$-systems \cite[Definition 4.13]{Walters}, which contains measure preserving transformations whose measure-theoretic entropy is in some sense completely positive, that is, the Kolmogorov-Sinai entropy of every nontrivial partition is positive, equivalently, the measure-theoretic entropy of every nontrivial measure-preserving factor is positive. It is natural to seek for an analog of this notion in topological dynamics. It turns out that the conditions characterizing $K$-systems in ergodic theory are no longer equivalent when translated to the topological setting. This problem was studied by Blanchard \cite{B92} who defined \emph{completely positive entropy} and \emph{uniform positive entropy}. A dynamical system $(X,T)$ has \emph{completely positive entropy} if all nontrivial topological factors of this system have positive topological entropy, and $(X,T)$ has \emph{uniform positive entropy} if $T$ has positive topological entropy with respect to every open cover of $X$ by two sets none of which is dense in $X$. 
Blanchard proved that uniform positive entropy implies completely positive entropy (this was also proved earlier by Bowen, see Proposition 4.2 in \cite{Bowen71}), but the converse implication is not true. Moreover, completely positive entropy does not imply any mixing property. Huang and Ye \cite{HY06} introduced the notion of a \emph{topological $K$-system}. Following them we say that $(X,T)$ is a \emph{topological $K$-system} if every finite cover of $X$ by nondense and open sets has positive topological entropy.

 The topological $K$-systems are also known as systems with \emph{uniform positive entropy of all orders}. In this nomenclature Blanchard's \emph{uniform positive entropy} is the \emph{uniform positive entropy of order $2$}. Every minimal topological $K$-system is mixing \cite{HSY05}. Huang and Ye \cite[Theorem 7.4]{HY06} observed that topological $K$-systems have a kind of a very weak specification property.

 Here we only mention an easy part of this connection (cf. \cite{HY06}, Theorem 7.4). It is easy to see that if a surjective system $(X,T)$ has
the specification property, then for any nonempty open sets $U_1,\ldots,U_k\subset X$ there is an $N$
such that for any $n\in\N$ and  $\phi\colon \set{0,\ldots,n}\to \set{1,\ldots,k}$ there is a point $z$
satisfying $T^{iN}(z)\in U_{\phi(i)}$ for $i=0,\ldots,n$. This immediately gives the following (cf. Proposition 21.6 in \cite{DGS} and Proposition 4.3 in \cite{Bowen71}).
\begin{thm}[folklore]
	If a surjective system $(X,T)$ has the specification property, then it is a topological $K$-system.
\end{thm}

The next result is a consequence for $d=1$ of Theorem B in \cite{EKW} (Eizenberg, Kifer and Weiss stated it for $\mathbb{Z}^d$ actions). Theorem B in \cite{EKW} asserts that if $(X,T)$ is an invertible dynamical system  with the specification property and $\mu$ is a $T$-invariant probability measure such that the function $\MT(X)\ni\nu\mapsto h_\nu(T)\in\R$ is upper semicontinuous at $\mu$, then $\mu$ is the limit in the weak$^\ast$ topology of a sequence of ergodic measures $\mu_n$ such that  the entropy of $\mu$ is the limit of the entropies of the $\mu_n$.
This is an important point in
obtaining large deviations estimates, which was first emphasized in \cite{FO} (see also \cite{Comman09, Yamamoto09}).
Analysis of the proof of Theorem B in \cite{EKW} yields the following.

\begin{thm}[cf. \cite{EKW}, Theorem B]\label{thm:8}
	Let $(X,T)$ be an invertible dynamical system with the specification property. Then the ergodic measures are \emph{entropy dense}, that is,
	for every measure $\mu\in \MT(X)$, every neighborhood $U$ of $\mu$ in $\MT(X)$ and every $\eps>0$ there is an ergodic measure $\nu\in U$ with
	$h_\nu(T)\in (h_\mu(T)-\eps,h_\mu(T)]$.
\end{thm}

Let $\Per_n(T)$ denote the set of fixed points of $T^n$, where $n\in\N$.
Observe that if $(X,T)$ is expansive, then for every $n$ the set $\Per_n(T)$ is finite, and is nonempty for all $n$ large enough provided that $(X,T)$ has the periodic specification property.

Bowen \cite{Bowen71} proved that if $T$ is expansive and has  periodic specification, then the topological entropy of $T$ equals the exponential growth rate of the number
of fixed points of $T^n$.
\begin{thm}[cf. \cite{Bowen71}, Theorem 4.5]
If $(X,T)$ is an invertible expansive dynamical system with the periodic specification property then
\[
\htop(T)=\lim_{n\to \infty}\frac{1}{n}\log |\Per_n(T)|.
\]
\end{thm}
Every expansive dynamical system has a measure of \textit{maximal entropy}, since expansiveness implies that the
function $\MT(X)\ni\mu\mapsto h_\mu(T)\in [0,\infty)$ is upper semicontinuous and every such function on a compact metric space is bounded from the above and attains its supremum. 
It turns out that for a system with the periodic specification property the entropy maximizing measure is unique and can be described more precisely.

For each $n\in \N$ such that $\Per_n(T)$ is nonempty denote by $\mu_n$ the probability measure uniformly distributed on $\Per_n(T)$, 
that is,
\begin{equation}
\mu_n=\frac{1}{|\Per_n(T)|}\sum_{x\in \Per_n(T)}\hat\delta(x).\label{mun}
\end{equation}
Clearly, each $\mu_n$ is an invariant measure. 
By the above observation, if $(X,T)$ is expansive and has the periodic specification property, then we can consider an infinite sequence formed by $\mu_n$'s.
The proof of following result may be found in \cite{DGS}. It closely follows
Bowen's proofs in \cite{Bowen71} and \cite{Bowen74}.

\begin{thm}[\cite{DGS}, Theorem 22.7]\label{thm:10}
If $(X,T)$ is an invertible expansive dynamical system with the periodic specification property, then the sequence $\mu_n$ defined by \eqref{mun} converges to a fully supported ergodic measure $\mu_B\in\MT(X)$, which is the unique measure of maximal entropy of $T$. In particular, $(X,T)$ is intrinsically ergodic.
\end{thm}

In Theorems~\ref{thm:8}--\ref{thm:10} one can replace invertible by surjective and expansiveness by positive expansiveness or expansiveness of the natural extension.

It is known that the set of fully supported measure is either empty or residual in $\MT(X)$, e.g. see \cite[Proposition~21.11]{DGS}.
It is easy to see that if minimal points are dense in $X$ then the set of fully supported measures is nonempty, hence
fully supported measures are dense in $\MT(X)$.
 It follows that the specification property has a strong influence not only on the topological entropy but also on the space of invariant measures. Sigmund studied relations between the specification property and the structure of $\MT(X)$ in \cite{Sigmund70,Sigmund74}. Parthasarathy \cite{Parthasarathy61} proved similar results for a dynamical system $(Y,T)$ where $Y=X^\infty$ is a product of countably many copies of a complete separable metric space $X$ and $T$ is the shift transformation.
Sigmund's results may be summarized as follows:
\begin{thm}[Sigmund] \label{thm:sigmund} If $(X,T)$ has the periodic specification property, then:
	\begin{enumerate}
		\item \label{Sig:1} The set $\MTp(X)$ is dense, hence $\MTe(X)$ is arcwise connected and residual 
		in $\MT(X)$, hence  $\MT(X)$ is the Poulsen simplex.
		\item \label{Sig:5} The set $\MTe(X)\cap\MTpos(X)$  is residual in $\MT(X)$.
		\item \label{Sig:6} The set $\MTmix(X)$ is of first category in $\MT(X)$.
		\item \label{Sig:9} The set of all non-atomic measures is residual in $\MT(X)$.
		\item \label{Sig:10} For every non-empty continuum $V\subset \MT(X)$ the set $\{x\in X : \hat{\omega}_T(x)=V\}$ is dense in $X$. In particular, every invariant measure has a generic point.
		\item \label{Sig:11} The set $\{x\in X : \hat{\omega}_T(x)=\MT(X)\}$ is residual in $X$.
		\item \label{Sig:12} The set of quasiregular points is of first category.
		\item \label{Sig:13} For every $l\in \N$ the set $\bigcup_{p=l}^\infty P(p)$
		is dense in $\MT(X)$, where $P(p)$ denotes the set of all invariant probability Borel measures supported on periodic points of period $p$.
		\item The set of strongly mixing measures is of first category in $\MT(X)$.
	\end{enumerate}
\end{thm}

There are various extensions of Sigmund's results. 
Hofbauer~\cite{Hofbauer87,Hofbauer88} and Hofbauer and Raith~\cite{HR} proposed weaker forms of the specification property to prove  a variant of Sigmund's Theorem for some transitive and not necessarily continuous transformations $T\colon [0,1]\to[0,1]$.
Further generalizations were given by Abdenur, Bonatti, Crovisier \cite{ABC11}, Coudene and Schapira \cite{CS}, Sun and Tian \cite{ST12} to name a few.

Entropy-density of ergodic measures implies that the ergodic measures are dense in the simplex of invariant measures, but there are systems with dense but not entropy-dense set of ergodic measures (see \cite[Proposition 8.6.]{GK}).

The paper \cite{GK} introduces two new properties of a~set $K\subset\text{Per}(T)$: \emph{closeability} with respect to $K$ and \emph{linkability} of $K$. It is proved there that Sigmund's Theorem holds for a~system which is closeable with respect to a linkable set $K\subset\text{Per}(T)$. The periodic specification property implies that the dynamical system is closeable with respect to $K=\text{Per}(T)$, which is also linkable. These methods lead to an~extension of Sigmund's theorem which covers also:
\begin{itemize}
 \item systems with the periodic weak specification property,
\item $\mathcal C^1$-generic diffeomorphisms on a manifold,
\item irreducible Markov chains over a~countable alphabet,
\item all $\beta$-shifts,
\item many other coded systems.
\end{itemize}
Furthermore, there is a~continuous-time counterpart of this theory. For the details we refer the reader to \cite{GK}.

There are many examples of systems with the specification property besides iterates of an Axiom A diffeomorphism restricted to an elementary set.
Weiss \cite{Weiss73} noted that a mixing sofic shift (hence a mixing shift of finite type) has the periodic specification property.
Kwapisz \cite{Kwapisz00} extended it to cocyclic shifts.


Blokh characterized the periodic specification property for continuous interval maps \cite{Blokh83,Blokh95} proving the following (an alternative proof was given by \cite{Buzzi97}):
\begin{thm}[\cite{Blokh83}, Theorem 6]
A 
dynamical system $([0,1],T)$ has the periodic specification property if and only if it is topologically mixing.
\end{thm}

Later, Blokh generalized this result to topological graphs \cite{Blokh84,Blokh87} (see also a~presentation of Blokh's work in \cite{ARR}). An independent proof, extending some ideas for interval case in \cite{Buzzi97} was developed in \cite{HKO}. Recall that a \emph{topological graph}
is a continuum $G$ such that there exists a~one-dimensional simplicial complex $\mathcal K$ with geometric carrier $|\mathcal K|$ homeomorphic to $G$ (see \cite[p.~10]{Croom78})\label{top-graph}. Examples include the compact interval, circle, all finite trees etc.

\begin{thm}[cf. \cite{Blokh87}, Theorem 1]
	Let $G$ be a topological graph.
	A dynamical system $(G,T)$ has the specification property if and only if it is topologically mixing.
\end{thm}

It would be interesting to know whether a similar result holds for dendrites.

We conclude this section by mentioning some important applications of the specification property we  have no place to describe in more details.
The specification property was used by Takens and Verbitskiy \cite{TV} to obtain a variational description of the dimension of multifractal decompositions. This result motivated Pfister and Sullivan \cite{PS07} to introduce the $g$-almost product property renamed later the almost specification property by Thompson \cite{Thompson12}. Another application is due to Fan, Liao and Peyri\`{e}re \cite{FPL08}, who proved that for any system with the specification property the Bowen's topological entropy of the set of generic points of any invariant measure $\mu$ is equal to the measure-theoretic entropy of $\mu$. Further generalizations can be found in \cite{Oliveira12, KT13, Varandas12}.




\section{Weak specification}

Among examples of dynamical systems with the periodic specification property are hyperbolic automorphisms of the torus.
Lind proved that non-hyperbolic toral automorphisms do not have the periodic specification property (see Theorem~\ref{thm:19}).
Nevertheless, Marcus showed that the periodic point measures are dense in the space of invariant measures for ergodic automorphisms of the torus (automorphisms which are ergodic with respect to the Haar measure on the torus). To apply Sigmund's ideas Marcus has extracted in \cite[Lemma 2.1]{Marcus80}, the following property and showed that it holds for every ergodic toral automorphism.

\begin{defn}\label{def:weak_spec}
A  dynamical system $(X,T)$ has the \emph{weak specification property} if for every $\eps>0$ there is a nondecreasing function $M_\eps \colon \N \to \N$ with $M_\eps(n)/n\to 0$ as $n\to \infty$ such that any $M_\eps$-spaced specification is $\eps$-traced by some point in $X$. We say that $M_\eps$ is an \emph{$\eps$-gap function} for $(X,T)$.
\end{defn}
Marcus  did not give this property any name in \cite{Marcus80}. It was coined \emph{almost weak specification} by Dateyama \cite{Dateyama82} (this name is also used by Pavlov \cite{Pavlov14} or Quas and Soo \cite{QS}). Dateyama chose this name probably due to the fact that at that time the term weak specification was used as a name for 
the property we call specification \cite{BS}. 
At present the \emph{almost specification property} (see below) has gained some attention, and as we explain later it is independent of the property given by Definition \ref{def:weak_spec}. Therefore we think that \emph{weak specification} is a more accurate name.

An easy modification of the above definition leads to the notion of the \emph{periodic weak specification property} in which we additionally require that the tracing point is periodic. As for the classical specification property, both weak specification notions are equivalent provided the natural extension is expansive. The proof is analogous to that of Lemma \ref{lem:psp}.
\begin{lem}[folklore]
If $(X,T)$ has the weak specification property and its natural expansion is expansive, then $(X,T)$ has the periodic weak specification property.
\end{lem}

Note that the length of a gap a tracing point is allowed to spend between two orbit segments of a specification depends on the length of the later segment, that is, in the definition of an $\nu$-spaced specification
we have the condition
\begin{equation}
a_i-b_{i-1} \ge  \nu(b_i-a_i+1)  \text{ for }2\le i \le n. \label{spaced:1}
\end{equation}
One may consider a ``dual'' notion of an $\nu$-spaced specification in which the length of a gap between two consecutive orbit segments in a specification
is a function of the length of the earlier segment, that is, we may replace the condition \eqref{spaced:1} by 
\begin{equation}
a_i-b_{i-1} \ge  \nu(b_{i-1}-a_{i-1}+1)  \text{ for }2\le i \le n. \label{spaced:2}
\end{equation}
It seems that there is no agreement which of those conditions should be used and both are present in the literature (the variant using \eqref{spaced:1} is used in \cite{Dateyama82, Dateyama91,KOR,QS} while \eqref{spaced:2} is required by \cite{Pavlov14}).
These two ``dual'' definitions of the weak specification property are non-equivalent, as shown by the example below. Nevertheless, the proofs assuming one of the variants seem to be easily adapted to the case when the other variant is used.
\begin{exmp}Let us call, tentatively, the weak specification property as defined in Definition \ref{def:weak_spec} the \emph{forward weak specification} property and its dual version (the one in which the condition \eqref{spaced:2} replaces \eqref{spaced:1}) the \emph{backward weak specification} property. We will construct two shift spaces (see Section~\ref{sec:symbolic} for definitions we use here).
Consider two sets of words over $\{0,1\}$ given by
\[\mathcal{F}=\set{1 0^b 1^a : a,b\in\N,\,b<\log_2 (a)}\text{ and }\mathcal{G}=\set{1^a 0^b 1 : a,b\in\N,\,b<\log_2 (a)}.
\]
Let $X=X_{\mathcal{F}}$ and $Y=X_{\mathcal{G}}$ be shift spaces defined by taking $\mathcal{F}$ and $\mathcal{G}$ as the sets of forbidden words. Note that for any words $u,w$ admissible in $X$ we have $u 0^{\lceil\log_2 |w|\rceil}w\in \lang(X)$, and similarly if $u,w\in\lang(Y)$, then $u 0^{\lceil\log_2 |u|\rceil}w$ is also admissible in $Y$. Using this observation it is easy to check that $(X,\sigma)$ satisfies the forward weak specification property and $(Y,\sigma)$ satisfies the backward weak specification property. Note that both shift spaces $X$ and $Y$ contain points $x_1=1^\infty$ and $x_2=01^\infty$. Thus, the words $1^\ell$ and $01^\ell$ are admissible in both $X$ and $Y$ for all $\ell\in\N$. Furthermore, the necessary condition for the word $1 w 01^{\ell}$ to be admissible in $X$ is that $w$ ends with $0^s$ where $s=\lfloor\log_2\ell\rfloor$.   Assume that $X$ has also the backward specification property. Let $k=M_{1/2}(1)$ where $M_{1/2}$ denotes the $1/2$-(``backward'')-gap function for $X$. This implies that for every $\ell\in\N$ there exists a word $w$ of length $k$ such that $1 w 01^{\ell}$ is  admissible in $X$. But this contradicts the definition of $X$ if $\log_2(\ell)\ge k+1$. Therefore $X$ cannot have the backward weak specification property. A similar argument shows that $Y$ does not have the forward weak specification property.
\end{exmp}

It is easy to see that weak specification is inherited by factors, finite products and higher iterates. Furthermore it implies topological mixing.

\begin{thm}[folklore]
\begin{enumerate}
\item If $(X,T)$ has the weak specification property then $(X,T^k)$ has the weak specification property for every $k\geq 1$.
\item If $(X,T)$, $(Y,S)$ have the weak specification property then $(X\times Y, T\times S)$ has the weak specification property.
\item Every factor of a system with the weak specification property has the weak specification property.
\item Every onto map $T\colon X\to X$ with the weak specification property is topologically mixing.
\end{enumerate}
\end{thm}
\begin{proof}
We prove only the last statement as the first three are obvious. Take $x,y\in X$ and $\eps>0$. It is enough to prove that for every $n\geq M_\eps(1)+1$ there exists $z\in X$ such that $\rho(x,z)<\eps$ and $\rho(T^n(z),y)<\eps$.
Fix any $n> M_{\eps}(1)$. Let $a_1=b_1=0$, $a_2=b_2=n$ and take any $y'\in T^{-n}(\set{y})$. Then $\set{T^{[a_1,b_1]}(x),T^{[a_2,b_2]}(y')}$
is an $M_\eps(1)$-spaced specification and hence the result follows.
\end{proof}




%

\subsection{Specification for automorphisms of compact groups}

Sigmund \cite[p. 287, Remark (E)]{Sigmund74} asked which ergodic automorphisms of compact groups have the specification property. Lind \cite{Lind79} gave the answer for ergodic toral automorphisms. The result of Marcus completed the characterization of specification-like properties for that case.
We will briefly describe these results below.

Lind \cite{Lind82} calls a toral automorphisms \emph{quasi-hyperbolic} if the associated linear map has no roots of unity as eigenvalues. An automorphisms of the torus is quasi-hyperbolic if and only if it is ergodic with respect to Haar measure \cite{Halmos43}. Quasi-hyperbolic toral automorphisms can be classified using the spectral properties of the associated linear maps. Following Lind \cite{Lind79} we distinguish:
\begin{itemize}
  \item \emph{Hyperbolic automorphisms}, that is, those without eigenvalues on the unit circle.
  \item \emph{Central spin automorphisms}, that is, those with some eigenvalues on the unit circle, but without off-diagonal $1$'s in the Jordan blocks associated with unitary eigenvalues.
  \item \emph{Central skew automorphisms}, that is, those with off-diagonal $1$'s in the Jordan blocks associated with some unitary eigenvalues.
\end{itemize}

We can summarize results of \cite{Lind82, Marcus80} as follows.

\begin{thm}[Lind, Marcus]\label{thm:19}
Let $T$ be a quasi-hyperbolic toral automorphisms. Then:
\begin{enumerate}
  \item $T$ has the periodic specification property if and only if $T$ is hyperbolic;
  \item $T$ has the specification property, but does not have the periodic specification property if and only if $T$ is central spin;
  \item $T$ has the weak specification property, but does not have the specification property if and only if $T$ is central skew.
\end{enumerate}
\end{thm}
Actually Marcus (see main theorem in \cite{Marcus80}) obtained a slightly stronger, periodic version of weak specification which allowed him to prove that for any quasi-hyperbolic toral automorphism $T$ the invariant measures supported on periodic points are dense in $\MT(X)$.

The above theorem shows that (periodic) specification and weak specification are different properties.

\begin{rem}
	Clearly, specification implies weak specification. We have explained above why the converse is not true.
\end{rem}

 Similar results hold for ergodic automorphisms of other compact metric groups. Here we mention only a result of Dateyama (see \cite{Dateyama91}) and refer the reader to references therein for more details and a more general statement for some nonabelian groups.

\begin{thm}[\cite{Dateyama91}, Corollary on p.345]
Let $X$ be a compact metric abelian group and $T$ be an automorphism of $X$. Then $(X, T)$ is ergodic with respect to Haar measure if and only if $(X, T)$ satisfies
weak specification.
\end{thm}


A dynamical system $(X,T)$ is called \emph{universal} if for every invertible, non-atomic, ergodic, and measure-preserving system $(Y,S,\mu)$ with the measure-theoretic entropy strictly less than the topological entropy of $T$ there exists a Borel embedding of $(Y, S)$ into $(X,T)$. It is \emph{fully universal} if one can, in addition, choose this embedding in such a~way that $\text{supp}(\mu^*)=X$, where $\mu^*$ denotes the push-forward of $\mu$. The Krieger theorem says that the full shift over a finite alphabet is universal.
Lind and Thouvenot \cite{LT78} proved that hyperbolic toral automorphisms are fully universal. This was recently extended by Quas and Soo, who proved the following theorem (we refer to \cite{QS} for terms not defined here).

\begin{thm}[\cite{QS}, Theorem 7] A self homeomorphism of a compact metric space is fully universal whenever it satisfies
\begin{enumerate}
  \item weak specification,
  \item asymptotic entropy expansiveness,
  \item the small boundary property.
\end{enumerate}
\end{thm}

Benjy Weiss (personal communication) has proved that the second assumption above (asymptotic entropy expansiveness) is not necessary. He also has a version of this result for $\Z^d$ actions. Universality of $\Z^d$-actions was also a subject of \cite{RS01}. 

\section{Almost specification}

Another specification-like notion is the \emph{almost specification property}. Pfister and Sullivan introduced it as the \emph{$g$-almost
product property} in \cite{PS05}. Thompson \cite{Thompson12} used a slightly modified definition and renamed it the \emph{almost specification property}.
$\beta$-shifts are model examples of dynamical systems with the almost specification property (see \cite{CT12,PS05}).
Here we follow Thompson's approach, hence the almost specification property presented below is
a priori weaker (less restrictive) than the notion introduced by Pfister and Sullivan.

\begin{defn}
We say that $g \colon \Zp\times(0,\eps_0)\to\N$, where $\eps_0>0$ is a \emph{mistake function} if for all $\eps<\eps_0$  and all
$n \in\Zp$  we have $g(n, \eps) \le g(n + 1, \eps)$ and
\[
\lim_{n\to \infty}
\frac{g(n, \eps)}{n}= 0.
\]
Given a mistake function $g$ we define a function $k_g\colon (0,\infty) \to \N$ by
declaring $k_g(\eps)$ to be the smallest $n\in\N$ such that $g(m,\eps)<m\eps$ for all $m\ge n$.
\end{defn}
\begin{defn}
Given a mistake function $g$, $0<\eps<\eps_0$ and $n\ge k_g(\eps)$ 
we define the set
\[
I(g; n, \eps) := \{\Lambda\subset \set{0,1,\ldots, n - 1} : \#\Lambda \ge n-g(n,\eps)\}.
\]
\end{defn}

We say that a point $y\in X$ $(g;\eps,n)$-traces an orbit segment $T^{[a,b]}(x)$ 
if for some $\Lambda\in I(g;n,\eps)$ we have
$\rho^T_\Lambda(T^a(x),T^a(y))\le\eps$. By $B_n(g;x,\eps)$ we denote the set of all points which
$(g;\eps,n)$-trace an orbit segment $T^{[0,n)}(x)$. 
Note that $B_n(g;x,\eps)$ is always closed and nonempty.

\begin{defn}
A dynamical system $(X,T)$ has the \emph{almost specification
property} if there exists a mistake function $g$ such that
for any $m\geq 1$, any $\eps_1,\ldots,\eps_m > 0$,
and any specification $\{T^{[a_j,b_j]}(x_j)\}_{j=1}^m$ with $b_j-a_j+1\ge k_g(\eps_j)$ for every $j=1,\ldots,m$
we can find a point $z\in X$ which $(g;b_j-a_j+1,\eps_j)$-traces the orbit segment
$T^{[a_j,b_j]}(x_j)$ for every $j=1,\ldots,m$.
\end{defn}
In other words,
the appropriate part of the orbit of $z$ $\eps_j$-traces with at most $g(b_j-a_j+1,\eps_j)$ mistakes the orbit of $x_j$ over $[a_j,b_j]$.

\begin{rem}
Pfister and Sullivan \cite[Proposition 2.1]{PS07} proved that the specification property implies the $g$-almost product property with any mistake function $g$. The proof can be easily adapted to show that the specification property implies the almost specification property. The converse is not true because for every $\beta>1$ the $\beta$-shift $X_\beta$ has the almost specification property with a mistake function $g(n)=1$ for all $n\in\N$ (see \cite{PS07}), while the set of $\beta>1$ such that $X_\beta$ has the specification property has Lebesgue measure zero \cite{Buzzi97,Schmeling97}. We recall that $\beta$-shifts are symbolic encodings of the $\beta$-transformations $x\mapsto \beta x \bmod 1$ on $[0,1]$. 
Given $\beta>1$ find a sequence $\{b_j\}_{j=1}^\infty$ with $0\le b_j <\beta$ such that
\[
1=\sum_{j=1}^\infty \frac{b_j}{\beta^j},
\]
where the $j$th ``digit'' of the above $\beta$-expansion of $1$ is given by
\[
b_j=\lfloor \beta\cdot T_\beta^{j-1}(1)\rfloor, \quad\text{ where}\quad T_\beta(x)=\beta x -\lfloor\beta x\rfloor =\beta x\bmod 1 \text{ for }x\in[0,1].
\]
If $\{b_j\}_{j=1}^\infty$ is not finite, that is, it does not end with a sequence of zeros only, then the $\beta$-shift is the set $X_\beta$ of all infinite sequences $x$ over the alphabet $\{0,1,\ldots,\lfloor \beta\rfloor\}$ such that $\sigma^k(x)<\{b_j\}_{j=1}^\infty$ lexicographically for each $k>0$.
If
\[\{b_j\}_{j=1}^\infty=i_1,\ldots,i_m,0,0,\ldots,\] then $x\in X_\beta$ if and only if
\[
\sigma^k(x)< i_1,\ldots,i_{m-1},(i_m-1), i_1,\ldots,i_{m-1},(i_m-1), i_1,\ldots
\]
lexicographically for each $k>0$ (see \cite{Parry60}).
This notion was introduced by R\'{e}nyi in \cite{R}. For more details see \cite{Blanchard89, Parry60, Thomsen05}.

\end{rem}

As noted above, the almost specification property of $(X,T)$ does not imply surjectivity of $T$. Furthermore, $(X,T)$ has the almost specification property
if and only if it has the same property when restricted to the measure center (see \cite[Theorem 6.7.]{WOC} or \cite[Theorem 5.1.]{KKO} for a proof). As a consequence, almost specification property alone does not imply any recurrence property like transitivity or mixing (see \cite{KKO}).
But the restriction of a system with the almost specification property to the measure center must be weakly mixing (see \cite{KKO}). We do not know whether one can conclude that almost specification implies mixing on the measure center.

Thompson \cite{Thompson12} used the almost specification property to study the irregular set of a dynamical system $(X,T)$. Given a continuous function $\varphi \colon  X \to \R$ we consider the \emph{irregular set} for $\varphi$ defined by
$$
\hat{X}(\varphi,T):= \left\{x\in X: \lim_{n\to \infty}\frac{1}{n}\sum_{i=0}^{n-1}\varphi(T^i(x)) \text{ does not exist}\right\}.
$$
Some authors call it the \emph{set of points with historic behaviour}. It is meant to stress that these points witness the history of the system and record the fluctuations, while points for which the limit exists capture only the average behaviour. The set $\hat{X}$ is the natural object of study of multifractal analysis. Although it is not detectable from the point of view of ergodic theory (it follows from Birkhoff's ergodic theorem that $\hat{X}$ is a universally null set) it can be large from the point of view of dimension theory. There is a vast literature on this topic, see \cite{EKL, Olsen03b, OW03} to mention only a few contributions.
Thompson's main result (see below) says that the irregular set of a system with the almost specification property is either empty or has full topological entropy. In this statement entropy is the Bowen's dimension-like characteristic of a non necessarily compact, nor invariant set $A\subset X$ denoted by $\htop(A,T)$ (see \cite[Definition 3.7]{Thompson12} or \cite{Pesin} for more details).

\begin{thm}[\cite{Thompson12}, Theorem 4.1]
Let $(X, T)$ be a dynamical system with the almost specification property. If a continuous function $\varphi \colon X\to \R$ satisfies
$$
\inf_{\mu \in \MT(X)}\int \varphi d\mu <\sup_{\mu \in \MT(X)}\int \varphi d\mu
$$
then $\htop(\hat{X}(\varphi,T),T)=\htop(T)$.
\end{thm}



\section{Almost and weak specification}

It is natural to ask whether the weak or almost specification property implies intrinsic ergodicity. Moreover, the definition of these properties might suggest that weak specification implies almost specification. The problem of intrinsic ergodicity of shift spaces with almost specification was mentioned in
\cite[p. 798]{CT12}, where another approach was developed in order to prove that certain classes of symbolic systems and their factors are intrinsically ergodic. It turns out that there are shift spaces with the weak (almost) specification property and many measures of maximal entropy. Moreover, there is no connection between the almost and the weak specification property. This was discovered  independently by Pavlov \cite{Pavlov14} and the authors of \cite{KOR}. In the latter paper there is a construction of a family of shift spaces, which contains:
\begin{enumerate}
               \item A shift space with the almost specification property and finite number of measures of maximal entropy concentrated on disjoint nowhere dense subsystems.
               \item A shift space with the weak specification  property and finite number of measures of maximal entropy concentrated on disjoint nowhere dense subsystems.
               \item A shift space with the almost specification property but without weak specification.
               \item Shift spaces $X$ and $Y$  satisfying
               \begin{enumerate}
               \item $Y$ is a factor of $X$,
                 \item their languages possess the Climenhaga-Thompson decomposition (see \cite{CT12})
               $\lang(X)=\CpX\cdot\G_X\cdot\CsX$ and $\lang(Y)=\CpY\cdot\G_Y\cdot\CsY$,
                 \item $h(\G_X)>h(\CpX\cup\CsX)$ and $h(\G_Y) < h(\CpY\cup\CsY)$,
                 \item $X$ is intrinsically ergodic, while $Y$ is not.
               \end{enumerate}
\end{enumerate}
This construction proves that the sufficient condition for the inheritance of intrinsic ergodicity by factors from the Climenhaga-Thompson paper \cite{CT12} is optimal --- if this condition does not hold, then the symbolic systems to which Theorem of \cite{CT12} applies may have a factor with many measures of maximal entropy. We refer the reader to \cite{CT12, KOR} for more details. It is also proved in \cite{KOR} that nontrivial dynamical systems with the almost specification property and a full invariant measure have uniform positive entropy and horseshoes (subsystems which are extensions of the full shift over a finite alphabet). Since $(X,T)$ has the almost specification property
if and only if it has the same property when restricted to the measure center (see \cite[Theorem 6.7.]{WOC} and \cite[Theorem 5.1.]{KKO}),
it follows that  minimal points are dense in the measure center, thus a minimal system with the almost specification property must be trivial.

It follows from \cite{Pavlov14, KOR} that for any positive nondecreasing function $f\colon\N\to\Zp$ with 
\[
\lim_{n\to \infty}\frac{f(n)}{n}=0 \quad \text{ and } \quad \liminf_{n\to\infty}\frac{f(n)}{\ln n} >0,
\]
there exists a shift space, which has the weak specification property with the gap function $f(n)$ and at least two measures of maximal entropy, whose supports are disjoint.
In \cite{KOR} it is shown that the same condition as for the gap function suffices for the existence of a shift space with the almost specification property, the mistake function $f$, and many measures of maximal entropy. Pavlov \cite{Pavlov14} proves that even a constant mistake function $g(n) = 4$ can not guarantee intrinsic ergodicity. He also shows that if the mistake or the gap function grows sufficiently slowly, then the shift cannot have
two measures of maximal entropy with disjoint supports. 
\begin{thm}[\cite{Pavlov14}, Theorems 1.3--4]
If a shift space $X$ has either
\begin{enumerate}
\item the weak specification property with the gap function $f$ satisfying
\[
\liminf_{n\to\infty}\frac{f(n)}{\ln n} =0, \text{ or}
\]
\item the almost specification property with the mistake function $g(n) = 1$,
\end{enumerate}
then it cannot have two measures of maximal entropy
with disjoint support.
\end{thm}

\section{Approximate product property}\label{sec:app}
Pfister and Sullivan \cite[Definition 4.2]{PS05} introduced the following weaker form of the specification property. \begin{defn}
We say that a dynamical system $(X,T)$ has the \emph{approximate product structure} if for any $\eps>0$, $\delta_1>0$ and $\delta_2>0$ there exists an integer $N>0$ such that for any $n\geq N$ and $\{x_i\}_{i=1}^{\infty} \subset X$ there are $\{h_i\}_{i=1}^{\infty}\subset\Zp$ and $y\in X$ satisfying $h_1=0$, $n\leq h_{i+1}-h_i\leq n(1+\delta_2)$ and $${ \left|\big\{0\leq j<n\,:\,\rho\big(T^{h_i+j}(y),T^j(x_i)\big)>\eps\big\}\right|\leq \delta_1} n\text{ for all } i\in\N.$$
\end{defn}
The thermodynamic behaviour of
a dynamical system with the approximate product structure is a consequence of the large scale structure of the orbit space of the system, which is essentially the product of weakly interacting large subsystems. 
Pfister and Sullivan refer to the notion of
an asymptotically decoupled probability measure introduced in \cite{Pfister02} in the context of statistical
mechanics as an inspiration for their definition.
They used almost product structure to obtain large deviations results, which
were previously proven for dynamical systems with the specification property in \cite{EKW}. They achieved it by proving first that the approximate product property is strong enough to imply entropy-density of ergodic measures.
\begin{rem}
It is clear that the weak (almost) specification property implies the approximate product property. We demonstrate below why neither converse is true.
\end{rem}


We observe that the approximate product property is equivalent to transitivity  for systems with the shadowing property. Thus every transitive system with shadowing is an example of a system with the approximate product property. Readers not familiar with the definition of the shadowing property will find it in the next section.

\begin{thm}\label{thm:shd:aps}
	Assume that $(X,T)$ has the shadowing property. The following conditions are equivalent:
	\begin{enumerate}
		\item $(X,T)$ is transitive,\label{shd:aps:1}
		\item $(X,T)$ has the approximate product property.\label{shd:aps:2}
	\end{enumerate}
\end{thm}
\begin{proof}
	First we prove $\eqref{shd:aps:2}\Longrightarrow\eqref{shd:aps:1}$. First we show that $T$ restricted to its measure center is transitive. Let $U$, $V$ be nonempty open subsets of $X$ with a nonempty intersection with the measure center. It follows that there are invariant measures $\mu_U$ and $\mu_V$ such that
$\mu_U(U)>0$ and $\mu_V(V)>0$. Pfister and Sullivan proved \cite[Theorem 2.1]{PS05}  that there is a sequence $\mu_n$ of ergodic measures weak$^*$ converging to $\mu=(1/2)(\mu_U+\mu_V)$.  By \cite[Proposition~2.7]{DGS} $\liminf_{n\to\infty}\mu_n(W)\geq \mu(W)$ for every open set $W$. Hence there is $m$ such that
$\mu_m(U)>0$, and $\mu_m(V)>0$. Since $\mu_m$ is ergodic, $T$ is transitive on $\supp \mu_m$.
This shows that $U\cap T^n(V)$ is nonempty for some $n\in\N$ and therefore $(X,T)$ is transitive on the measure center.

Note that if $T$ has the shadowing property, then minimal points are dense in $\Omega(T)$ (see \cite[Corollary 1(i)]{Moothathu11}). On the other hand
the measure center always contains the closure of the set of minimal points and is contained in $\Omega(T)$. Hence shadowing implies that the measure center coincides with the non-wandering set. But shadowing implies also that $\CR(T)=\Omega(T)$ (\cite[Theorem 3.1.2.]{AH}) therefore the system $(\CR(T),T)$ is transitive. By  \cite[Proposition 5]{TShi} (see also \cite{Mazur00}) if $(\CR(T),T)$ is transitive, then $\CR(T)=X$. We conclude that $(X,T)$ is also transitive.
	
	For the proof of $\eqref{shd:aps:1}\Longrightarrow\eqref{shd:aps:2}$ fix
	any $\eps,\delta_1,\delta_2>0$. Use shadowing to pick $\delta>0$ for the given $\eps$. There is a finite cover $\{U_1,\ldots, U_p\}$ of $X$ with the diameter smaller than $\delta$. Since $T$ is transitive, for every pair $i,j\in\{1,\ldots, p\}$ there exists $\alpha(i,j)\in\N$ and a point $z_{i,j}\in U_i$ such that $T^{\alpha(i,j)}(z_{i,j})\in U_j$. Let $M=\max\limits_{i,j}\{\alpha(i,j)\}\in\N$. Let $N$ be such that $M\leq N\delta_2$. We claim that it is enough to set $N(\eps,\delta_1,\delta_2)=N$. Fix $\s{x}\in X^{\infty}$ and $n\geq N$.
	For every $i\in\N$ let $k(i)\in\{1,\ldots, p\}$ be such that $x_i\in U_{k(i)}$ and $l(i)\in\{1,\ldots, p\}$ satisfy $T^{n}(x_i)\in U_{l(i)}$.
	
	Define a $\delta$-pseudo-orbit as follows:
	\begin{multline*}
	\big(x_1,\,T(x_1),\,\ldots,\, T^{n-1}(x_1),\quad  z_{l(1),k(2)},\, T( z_{l(1),k(2)}),\ldots,\, T^{\alpha(l(1), k(2))-1}( z_{l(1),k(2)}),\quad\\
	x_2,\,T(x_2),\,\ldots, \,T^{n-1}(x_2),\quad  z_{l(2), k(3)},\,\ldots\big).
	\end{multline*}
	There exists $y\in X$ which $\eps$-traces it. Setting $$h_i=(i-1)n+\sum_{j=1}^{i-1}\alpha\big(l(j), k(j+1)\big),$$ we get that for all $i\in\N$ one has $T^{h_i}(y)\in B_n(x_i,\eps)$.
	Moreover $h_1=0$ and for any $i\in\N$ we have $h_{i+1}-h_i=n+\alpha\big (l(i),k(i+1)\big)\in\big[n, n(1+\delta_2)\big]$.
	This completes the proof.
\end{proof}

As a corollary we obtain the following theorem, which  is an extension of a result of Li and Oprocha \cite{LO13} who proved that for weakly mixing systems with shadowing the ergodic measures supported on orbit closures of regularly recurrent points are dense in the simplex of all invariant Borel probability measures. We obtain a stronger conclusion of entropy density of ergodic measures under a weaker assumption of transitivity and shadowing, but we do not know whether the measures supported on the orbit closures of regularly recurrent points are dense in this more general situation.

\begin{cor}\label{cor:shad}
If $(X,T)$ is transitive and has the shadowing property, then the set $\mathcal M_T^e(X)$ is entropy-dense in $\mathcal M_T(X)$.
\end{cor}

\begin{exmp}
Every  Axiom A diffeomorphism and every transitive shift of finite type has the shadowing property (see \cite{AH} and \cite{WaltersSFT}).
\end{exmp}

\begin{exmp}\label{ex:adding}
We briefly recall the construction of the \emph{adding machine}.
Equip $\Sigma=\{0,1\}^\infty$ with the product topology. Define the addition $\oplus$ $$\Sigma\times\Sigma\ni (\alpha,\beta)\mapsto \alpha\oplus\beta \in\Sigma$$ as the coordinate-wise addition modulo $2$ with possible infinite carry-over (see \cite{Vries}, p.~246 for details).
Let $\tau \colon\Sigma\to\Sigma$ be given by $\tau(z)=z\oplus \mathbf{1}$, where $\mathbf{1}=(1,0,0,\ldots)$.

The dynamical system $(\Sigma,\tau)$ is known as the \emph{dyadic adding machine} and has the shadowing property (see \cite{Kurka03}). Moreover $(\Sigma,\tau)$ is minimal, equicontinuous, uniquely ergodic, transitive but not totally transitive and has zero topological entropy (see \cite{Vries}, Chapter III, (5.12) 3). As a consequence of Corollary~\ref{cor:shad} we obtain that $(\Sigma, \tau)$ has the approximate product structure, but it can have neither weak nor almost specification.

By the same argument, $(\Sigma,\tau^2)$ does not have approximate product structure since it has the shadowing property, but is not transitive.
\end{exmp}

The above example shows that approximate product structure does not imply weak mixing, nor positive topological entropy, nor is inherited by Cartesian products. Moreover, it demonstrates that $(X,T)$ may have the approximate product structure while $T^k$ does not have this property for some $k\ge 2$.


\begin{thm}[\cite{PS05}, Proposition 2.2]
Any factor of a system with the approximate product structure has the approximate product structure.
\end{thm}

\begin{rem}
If the set of ergodic measures is dense in the simplex of all invariant measures, then the simplex is either trivial or the Poulsen simplex.  Example~\ref{ex:adding} above shows that the approximate product property cannot guarantee that the simplex of invariant measures is Poulsen.
\end{rem}

\subsection{Dynamical properties relative to a regular periodic decomposition}
For shifts of finite type or interval maps the periodic specification property is equivalent to topological mixing. Therefore transitive, but not mixing shifts of finite type and interval maps do not have specification, nor weak (almost) specification as the later two properties imply weak mixing which is in this case equivalent to mixing. But we will show that these and similar examples have the approximate product property.

The domain of a transitive map $T\colon X\to X$ cannot be decomposed into $T$-invariant topologically nontrivial subdomains (sets with pairwise disjoint nonempty interior). Banks studied in \cite{Banks97} transitive maps $T$ such that $T^n$ is not transitive for some integer $n\ge 2$. He showed that for such maps there exists a decomposition of $X$ into topologically nontrivial subsets $D_0,D_1,\ldots,D_{n-1}$  which are mapped by $T$ in a periodic fashion, that is, $T(D_i)=D_{(i+1)\bmod n}$ for $0\le i\le n-1$ and fulfilling some additional assumption. He called these decompositions \emph{regular periodic decompositions}.

\begin{defn}
We say that a collection $\mathcal{D}=\{D_0,\ldots,D_{n-1}\}$ is a~\emph{regular periodic decomposition} of a dynamical system $(X,T)$ if the following conditions are satisfied:
\begin{enumerate}
\item for each $i\in\{0,\ldots, n-1\}$ the set $D_i$ is regular closed (that is $D_i=\overline{\text{int}(D_i)}$),
\item $D_i\cap D_j$ is nowhere dense whenever $i\neq j$,
\item $T(D_i)\subset D_{(i+1)\bmod n}$ for $0\leq i<n-1$,
\item $D_0\cup\ldots\cup D_{n-1}=X$.
\end{enumerate}
\end{defn}
Clearly, $\{X\}$ is always a regular periodic decomposition of the space $X$. We call such a~decomposition {\it trivial}. Banks proved that a transitive dynamical system is either totally transitive, or it has a regular periodic decomposition $\mathcal{D}=\{D_0,\ldots,D_{n-1}\}$ for some $n\ge 2$.

A class $P$ of compact dynamical systems is a \emph{property} if it is saturated with respect to conjugacy, that is, if $(X,T)\in P$ and $(Y,S)$ is conjugated to $(X,T)$, then $(Y,S)\in P$.

\begin{defn}\label{df:relativespec}
Let $P$ be a~property of compact dynamical systems (e.g. transitivity, (weak) mixing, specification).
A~dynamical system $(X,T)$ has the~{\it property $P$ relative to a~regular periodic decomposition $\mathcal D=\{D_0,\ldots, D_{n-1}\}$} if $T^n|_{D_i}$ has the property $P$ for each $i\in\{0,\ldots,n-1\}$. We say that $(X,T)$ has the \emph{relative property $P$}, if there exists a regular periodic decomposition $\mathcal D$ such that $(X,T)$ has the relative property $P$ with respect to $\mathcal D$.
\end{defn}
\begin{rem}
If $(X,T)$ has the property $P$ and $\tau$ is a cyclic permutation of the set $\{0,\ldots, r-1\}$ given by $\tau(i)=i+1\bmod r$, then the system
\[\big(X\times\{0,\ldots, r-1\},\, S\big), \text{ where }S(x,i)=
\begin{cases} (T(x),1),&\text{ if }i=0,\\
(x,\tau(i)),&\text{ otherwise}
\end{cases}
\]
has the relative property $P$ with respect to the regular periodic decomposition $\{D_0,\ldots, D_{r-1}\}$ where $D_i=X\times\{i\}$ for every $i\in\{0,\ldots, r-1\}$.
\end{rem}

Hence we can consider systems with the relative specification-like property.
It is a simple but a bit surprising fact that the almost product property and its relative version are equivalent.

\begin{thm}\label{thm:40}
A dynamical system $(X,T)$ has the relative approximate product structure if and only if it has the approximate product structure.
\end{thm}
\begin{proof}
Choose a regular periodic decomposition $\mathcal D=\{D_0,\ldots, D_{r-1}\}$ such that $T^r|_{D_i}$ has the approximate product structure for every $i\in\{0,\ldots, r-1\}$. Fix $\eps>0$, $\delta_1>0$ and $\delta_2>0$. Let $\eta\in(0,\eps)$ be such that for every $y,z\in X$ and every $i=0,\ldots, r$ one has $\rho(T^i(y), T^i(z))\leq\eps$ provided $\rho(y,z)\leq\eta$. Let $M\geq 4$ be chosen for $\eta,\, \delta_1/2,\, \delta_2/2$ using the approximate product structure of $T^r|_{D_0}$.
Let $N\geq rM$ be such that $\delta_2\geq 6r/(N-2r)$. We claim that it is enough to set $N(\eps, \delta_1, \delta_2)=N$. Fix $n\geq N$ and $\{x_i\}_{i=1}^{\infty}\in X^{\infty}$. For any $i\in\N$ let $\tilde x_i\in D_0$ be such that there exists $p_i\in\{0,\ldots, r-1\}$ such that $T^{p_i}(\tilde x_i)=x_i$. There are $m> M$ and $q\in\{0,\ldots,r-1\}$ satisfying $n=(m-1)r-q$. Let $y\in D_0$ and $\{h_i\}_{i=1}^{\infty}\subset\N$ be such that $h_1=0$, $m\leq h_{i+1}-h_i\leq (1+\delta_2/2)m$ for every $i\in\N$ and $$ \left|\big\{0\leq j<m\,:\,\rho\big(T^{r(h_i+j)}(y),T^{rj}(\tilde x_i)\big)>\eta\big\}\right|\leq \delta_1 m/2\text{ for any } i\in\N.$$
By the choice of $\eta$ we get that
$$ \left|\big\{0\leq j<rm\,:\,\rho\big(T^{rh_i+j}(y),T^j(\tilde x_i)\big)>\eps\big\}\right|\leq \delta_1rm/2\text{ for any } i\in\N.$$
Consequently, $$ \left|\big\{0\leq j<r(m-1)\,:\,\rho\big(T^{rh_i+j+p_i}(y),T^j( x_i)\big)>\eps\big\}\right|\leq \delta_1rm/2\text{ for any } i\in\N.$$
Set $z=T^{p_1}(y)$, $g_1=0$ and $g_i=rh_i+p_i$ for $i>1$.
One has $g_{i+1}-g_i=r(h_{i+1}-h_i)+p_{i+1}-p_i\in [r(m-1), (1+\delta_2/2)rm+r]\subset[n,(1+\delta_2)n]$ for any $i>1$, where the inclusion holds because $\delta_2\geq 6r/(n-2r)$ and hence $(1+\delta_2/2)rm+r\leq(1+\delta_2)((m-1)r-q)$ .
Moreover for every $i\in\N$ we have
$$ \left|\big\{0\leq j<n\,:\,\rho\big(T^{g_i}(z),T^j( x_i)\big)>\eps\big\}\right|\leq \delta_1rm/2\leq \delta_1 n\text{ for any } i\in\N$$
since $m\geq 4$.
This completes the proof.
\end{proof}

We recall that a \emph{topological graph} (or a \emph{graph} for short) is, informally said, a~compact connected metric space homeomorphic to a representation of a graph (a~combinatorial object consisting of a finite set of vertices and a finite set of edges joining pairs of distinct vertices) in the Euclidean space, where the vertices of the graph are represented by distinct points and the edges are disjoint arcs joining the corresponding pairs of points (see page \pageref{top-graph} or \cite[p. 10]{Croom78}).

\begin{cor}\label{cor:41}
If $(X,T)$ has the relative almost (weak) specification property, then  it has the approximate product property. In particular,  transitive and noninvertible graph map or transitive sofic shift has the approximate product property.
\end{cor}
\begin{proof}
The first part is a consequence of the previous theorem. The second part follows from the well known fact that transitive noninvertible graph maps and  transitive sofic shifts have the relative specification property. Note that we added the noninvertibility assumption for graph maps to exclude the irrational rotation of the circle. It follows from \cite{Blokh84, Blokh87} (see also \cite{Banks97}) that this is the only possible example of a transitive graph map without relative specification.
\end{proof}

\section{Specification and shadowing}\label{sec:shadowing}

One of fundamental tools of topological dynamics is the shadowing property (or pseudo-orbit tracing property), which allows tracing pseudo-orbits (approximate orbits, that is, sequences where the next point is uniformly close to the image of the previous point) with real orbits. We recall that a sequence $\s{x}\in X^\infty$ is a \emph{$\delta$-pseudo-orbit} for $T$ if
$\rho(T(x_n),x_{n+1})<\delta$ for each $n\in\Zp$. We say that a dynamical system $(X,T)$ has the \emph{shadowing property} if for any $\eps>0$ there is a $\delta>0$ such that for every $\delta$-pseudo-orbit $\s{x}$ one can find a point $y\in X$ with $\rho(x_n,T^n(y))<\eps$ for all $n\in\Zp$. 

Bowen defined specification for systems with the shadowing property. 
Subsequent generalizations of the specification property were defined mostly for systems without the shadowing property. It is perhaps an interesting phenomenon that if $(X,T)$ has the shadowing property, then many specification-like properties are equivalent.

There are many notions generalizing the shadowing property. Here we are concerned with two averaged versions of shadowing. Both follow from the almost specification property (see \cite{KKO,WOC}).

The first was introduced by Blank \cite{Blank88}, who considered sequences $\s{x}$ in $X$ in which the distances $\rho(T(x_n),x_{n+1})$ are small only on average and points whose orbits trace such sequences with small average errors.
\begin{defn}
A sequence $\seq{x}{n}\in X^{\infty}$ is a \emph{$\delta$-average-pseudo-orbit} for $T$ if there is an integer $N>0$ such that for every $n>N$ and $k\geq 0$  one has \[\frac{1}{n}\sum_{i=0}^{n-1} \rho(T(x_{i+k}),x_{i+k+1})<\delta.\]
\end{defn}

\begin{defn}
A dynamical system $(X,T)$ has the \emph{average shadowing property} if for every $\eps>0$ there is a $\delta>0$ such that for any $\delta$-average-pseudo-orbit $\seq{x}{n}$ for $T$ there is $y\in X$ with
\[
\limsup_{n\to\infty}\frac{1}{n}\sum_{k=0}^{n-1} \rho(T^k(y),x_{k})<\eps.
\]
\end{defn}

The next notion, coined by Gu \cite{Gu07}, appeared earlier unnamed in \cite[Remark 3]{Sigmund77}.

\begin{defn}
A sequence $\seq{x}{n}\in X^{\infty}$ is an {\it asymptotic-average-pseudo-orbit} for $T$ if
$$\lim\limits_{N\to\infty}\frac{1}{N}\sum_{i=0}^{N-1} \rho(T(x_i),x_{i+1})=0.$$
\end{defn}

Gu \cite{Gu07} also introduced the following generalization of the shadowing property.
\begin{defn}
A dynamical system $(X,T)$ has the \emph{asymptotic average shadowing property} if for every asymptotic-average-pseudo-orbit $\seq{x}{n}$ for $T$ there is $y\in X$ with
\[
\limsup_{n\to\infty}\frac{1}{n}\sum_{k=0}^{n-1} \rho(T^k(y),x_{k})=0.
\]
\end{defn}

The following fact summarizes relations between various mixing properties under assumption of shadowing.
Similar results were implicitly used before by many authors, starting with Bowen's classical work \cite{Bowen71}.
A complete and purely topological proof of most of the implications stated below can be found in \cite{KO12} which is based on an earlier paper by Sakai \cite{Sakai03}. The remaining implications are proved in \cite{KKO} (see also \cite{KO10}).

	\begin{thm}[many authors]\label{thm:shadowing}
		Let $(X,T)$ be a dynamical system with the
		shadowing property. Then the following conditions are equivalent:
		\begin{enumerate}
			\item\label{main:c1} $(X,T)$ is totally transitive,
			\item\label{main:c2} $(X,T)$ is topologically weakly mixing,
			\item\label{main:c3} $(X,T)$ is topologically mixing,
			\item\label{main:c4} $T$ is onto and $(X,T)$ has the specification property,
\item\label{main:c44} $T$ is onto and $(X,T)$ has the asymptotic average shadowing property,			
\item\label{main:c5} $T$ is onto and $(X,T)$ has the average shadowing property,
\item\label{main:c6} $T$ is onto and $(X,T)$ has the almost specification property.
		\end{enumerate}
		Moreover, if  the natural extension $(X_T,\sigma_T)$ is expansive, then any of the above conditions
		is equivalent to the periodic specification property of $T$.
	\end{thm}

Note that if $(X,T)$ is expansive or positively expansive, then the natural extension $(X_T,\sigma_T)$ is expansive (see \cite[p.~57]{AH}). As we observed above the approximate product property is equivalent to transitivity for systems with the specification property.
\section{Symbolic dynamics}\label{sec:symbolic}

We encourage the reader unfamiliar with techniques from symbolic dynamics to consult \cite{LM}. Here we follow the terminology and notation of \cite{LM}
as close as possible. We restrict our presentation to one-sided shifts, but all results presented here remain true in the two-sided setting.

Equip a finite \emph{alphabet} $\alf$ with discrete topology and consider $\alf^\infty$ as a compact metric space in the product topology (recall that $\alf^\infty$ denotes the set of all infinite sequences of elements of $\alf$ indexed by nonnegative integers). The formula $\rho(x,y)=2^{-k}$, where $x,y\in\alf^\infty$ and $k=\sup\{j\ge 0:x_j=y_j\}$ generates the topology of $\alf^\infty$ (we agree that $2^{-\infty}=0$ here).

Let $\sigma$ be the \emph{shift} $\{x_i\}_{i=0}^\infty\mapsto \{x_{i+1}\}_{i=0}^\infty$ on $\alf^\infty$.
A \emph{shift space} over $\alf$ is a nonempty closed and $\sigma$-invariant subset of $\alf^\infty$. We call a dynamical system $(\alf^\infty,\sigma)$ the \emph{full shift} over $\alf$.

A \emph{block} of length $k$ over $\alf$ is any string $w = w_1w_2\ldots w_k$ of symbols from $\alf$. We shall use the term ``\emph{a word}'' and ``\emph{a block}'' interchangeably.
The length of a word $w$ is denoted by $|w|$.
We say that a block $w$ \emph{occurs} or \emph{appears} in $x$ if $w = x_ix_{i+1}\ldots x_j$ for some $0\le i\le j$. The set of all words that occur in $x\in\alf^\infty$ is denoted $\lang(x)$. The \emph{language} of a shift space $X$ is the set $\lang(X)$ of all blocks that occur in some $\lang(x)$ for $x\in X$. By $\lang_n(X)$ we denote the set of blocks of length $n$ in $\lang(X)$. Concatenation of blocks or sequences is indicated by juxtaposition in the obvious way, for example $w^n=w\ldots w$ ($n$-times) and $w^\infty=www\ldots\in\alf^\infty$.

There is a characterization of a shift spaces using \emph{forbidden blocks}. It says that $X\subset \alf^{\infty}$ is a shift space if and only if there exists a set $\mathcal{F}$ consisting of blocks over $\alf$ such that $x\in X$ is equivalent to $\lang(x)\cap\mathcal{F}=\emptyset$ (in other words: no block from $\mathcal{F}$ occurs in $x$). A \emph{shift of finite type} is a shift space which can be defined using a finite set of forbidden blocks.

By a \emph{countable graph} we mean an irreducible directed graph with at most countably many vertices and edges. A graph is \emph{irreducible} if for any pair of its vertices $(v_i, v_j)$ there is a directed path from $v_i$ to $v_j$.

A \emph{labeled graph} $(G,\Theta)$ is a countable graph $G$ together with a \emph{labeling} function $\Theta$ mapping edges of $G$ to a finite alphabet $\alf$. 
The set $Y_G$ of infinite sequences constructed by reading off labels along an infinite path of $(G,\Theta)$ is
shift invariant, thus its closure $X=\overline{Y_G}$ in $\mathcal{A}^{\infty}$ is a shift space. Then we also say that $X$ is \emph{presented} by $(G,\Theta)$. Any shift space admitting such a presentation is a \emph{coded system}. A \emph{sofic shift} is a coded system which can be presented by a finite graph. A \emph{synchronized system} is a shift space which has a \emph{synchronizing word}, that is, there is $v\in\lang(X)$ such that $uv,vw\in\lang(X)$ imply $uvw\in\lang(X)$. Synchronized systems were introduced in \cite{BH}. Every synchronized system is coded.

The uniqueness of minimal right-resolving presentation known for sofic shifts
extends to synchronized systems as outlined in \cite[p. 451]{LM} (see  also \cite[p. 1241]{Thomsen} and references therein). Synchronized systems and their generalizations were extensively studied in \cite{FF}.

A \emph{cylinder set} of a word $u\in\lang_{r}(X)$ in a shift space $X$, where $r\in\N$, is the set $[u]\subset X$ of points from $X$ which begin with the block $u$, that is, $\{y\in X: y_{[0,r-1]}=u\}$. Cylinders are open and closed subsets of $X$ generating the topology. Furthermore, if $x\in X$, then the Bowen balls $B_n(x, \eps)$ centered at $x$ coincide with the cylinder sets $[x_{[0,k]}]$, where $k=n+s$ and $s$ depends only on $\eps>0$.

Using this terminology the definitions of dynamical properties can be conveniently restated for shift spaces: 
\begin{enumerate}
  \item A shift space $X$ is \emph{transitive} if for any $u,v\in \lang(X)$ there is $w\in \lang(X)$ such that $uwv\in \lang(X)$;
  \item A shift space $X$ is \emph{totally transitive} if for any $u,v\in \lang(X)$ and any $n>0$ there is $w\in \lang(X)$ such that $uwv\in \lang(X)$ and $n$ divides $|uw|$;
  \item A shift space $X$ is \emph{weakly mixing} if for any $u_1,v_1,u_2,v_2\in \lang(X)$ there are $w_1,w_2\in \lang(X)$ such that $u_1w_1v_1,u_2w_2v_2\in \lang(X)$ and $|u_1w_1|=|u_2w_2|$;
  \item A shift space $X$ is \emph{mixing} if for every $u,v\in \lang(X)$ there is $N>0$ such that for every $n>N$ there is $w\in \lang_n(X)$ such that $uwv\in \lang(X)$.
  \item A shift space $X$ has the \emph{specification property} if there is an integer $N\geq 0$ such that for any $u,v\in \lang(X)$ there is $w\in \lang_N(X)$ such that $uwv\in \lang(X)$.
\item We say that a nondecreasing function $\theta \colon \Zp\to\Zp$ is a mistake function if $\theta(n) \le n$ for all $n$ and $\theta(n)/n \to 0$.
A shift space has the \emph{almost specification property} if there exists a mistake function $\theta$ such that for every $n\in\N$ and $w_1,\ldots,w_n\in\lang(X)$, there exist words $v_1,\ldots,v_n\in\lang(X)$ with $|v_i| = |w_i|$ such that $v_1v_2\ldots v_n\in\lang(X)$ and each
$v_i$ differs from $w_i$ in at most $\theta(|v_i|)$ places.
\item A shift space $X$ has the \emph{weak specification property} if for every $n\in \N$ there exists $t(n)\in\N$ such that $t(n)/n\to 0$ as $n\to\infty$ and any $u,w \in \lang(X)$ with $|w|=n$ and $k\geq t(n)$ there exists a word $v\in \lang_k(X)$ such that $x = uvw \in\lang(X)$.
\item A~shift space $X$ has the \emph{variable specification property} if there exists $N\in\N$ such that for all $u,v\in\lang(X)$, there exists $w\in\lang(X)$ with $uwv\in\lang(X)$ and $|w|\leq N$.
\item A shift space $X$ has the \emph{strong property $P$} if for any $k\ge 2$
 and any words $u_1,\ldots, u_k\in \lang(X)$ with $|u_1|=\ldots=|u_k|$ there is an $n\in\N$ such that for any $N\in\N$ and function
 $\varphi\colon\{1,\ldots,N\}\to\{1,\ldots,k\}$ there are
 words $w_1,\ldots, w_{N-1}\in \lang_n(X)$ such that $u_{\varphi(1)}w_1u_2\ldots u_{\varphi(N-1)}w_{N-1}u_{\varphi(N)}\in \lang(X)$;
\end{enumerate}
Blanchard \cite{B92} defined the strong property $P$, which is also a some form of specification  and proved that it implies uniformly positive entropy and thus weak mixing, and does not imply mixing. A simpler example of this kind is provided in \cite{EKO}. Note that all unilateral symbolic dynamical systems are positively expansive (two-sided shift spaces contained in $\alf^\Z$ are expansive), thus the specification property and the periodic specification property are equivalent.
Also the weak specification property and the periodic weak specification property are equivalent in shift spaces.
Thompson's paper \cite{Thompson12} concludes with an example showing that the almost specification property does not imply its periodic variant even for expansive or positively expansive systems.

The variable specification property was introduced by Jung \cite{Jung11} under the name \emph{almost specification property}. Jung \cite[Lemma 3.7]{Jung11} also proved that every shift space $X$ with the variable specification property is synchronized, and $X$ has the variable specification property if and only if $X$ has the relative specification property with respect to  a regular periodic decomposition. Bertrand-Mathis \cite{Bertrand88} was first to prove that the specification property for shift spaces implies the existence of a~synchronizing word.

Thompson's shift mentioned above does not have periodic points, hence it cannot be coded. We do not know whether there are shift spaces with the weak specification property which are not coded.

Buzzi \cite{Buzzi05} investigated a class of shifts which arose in symbolic coding of several classes of non-uniformly hyperbolic dynamical systems such as entropy-expanding maps. It would be interesting to study the specification-like properties in this setting.

\begin{defn}
Given a~(possibly finite) set $S=\{n_1, n_2,\ldots\}\subset\Zp$ with $n_i\le n_{i+1}$ let $\mathcal{F}_S=\{10^t1\in\{0,1\}^+: t\notin S\}$ and let $X(S)$ be the shift space defined by declaring $\mathcal{F}_S$ as the set of forbidden words. Then $X(S)$ is a synchronized system called an \emph{$S$-gap shift}.
\end{defn}

In \cite[Example 3.4]{Jung11} there is a following characterization of specification properties for an $S$-gap shift $X(S)$, where $S=\{n_1, n_2\ldots\}\subset\Zp$  and $n_i\le n_{i+1}$.

\begin{enumerate}
  \item $X(S)$ has the variable specification property if and only if  $\sup_i\left|n_{i+1}-n_i\right|<\infty$,
  \item $X(S)$ is mixing if and only if $\gcd\{n+1: n\in S\}=1$,
  \item $X(S)$ has the periodic specification property if and only if $\gcd\{n+1: n\in S\}=1$ and $\sup_i\left|n_{i+1}-n_i\right|<\infty$.
\end{enumerate}
Using the above observation it is easy to show that there are shifts spaces with the relative specification property but without specification, and there are synchronized and mixing shift spaces without any form of specification.

A (proper) generalization of sofic shifts was introduced by Kwapisz in \cite{Kwapisz00}. He was motivated by certain computations important for the theory of cohomological Conley index. Let $\alf$ be an alphabet, fix any $m\in \N$ and consider
a family of square matrices with integer entries $\Phi=\set{\Phi_a}_{a\in \alf}$ indexed by $\alf$. If we denote by $0$ the zero matrix, then a \emph{cocyclic subshift} of $\Phi$
is a shift space
\[
X_{\Phi}=\set{x\in \alf^\infty : \Phi_{x_0}\Phi_{x_1}\ldots \Phi_{x_n}\neq 0,  \text{ for all } n\in \Zp}.
\]
We recall them because from the point of view of specification-like properties they behave much like sofic systems.

The following theorem summarizes connections between variants of specification for coded systems.
The main ingredient of the proof are two equivalences: equivalence of weak mixing and mixing, and mixing and specification. Given these two facts the rest is more or less standard. For shifts of finite type it follows mostly from Bowen's work, Weiss noted that it holds for sofic shifts and Kwapisz \cite{Kwapisz00} proved it for cocyclic shifts. 
\begin{thm}[many authors]\label{thm:coded_compilation}
Let $X$ be a non trivial coded system. Then the following conditions
are equivalent:
\begin{enumerate}[(a)]
\item\label{mix} $X$ is topologically mixing;
\item\label{pp} $X$ has the strong property $P$;
\item\label{wm} $X$ is topologically weakly mixing;
\item\label{tt} $X$ is totally transitive.
\end{enumerate}
If $X$ is synchronized, then any of the above conditions is equivalent to
\begin{enumerate}[(a)]\setcounter{enumi}{4}
\item\label{rel-p} $X$ has two periodic points with relatively prime primary periods.
\end{enumerate}
Moreover, there exists a coded system $X$ fulfilling \eqref{mix}--\eqref{tt}, but not \eqref{rel-p}.
If $X$ is cocyclic (in particular, if $X$ is sofic or of finite type), then any of the above conditions is equivalent to
\begin{enumerate}[(a)]\setcounter{enumi}{5}
\item\label{spec} $X$ has the periodic specification property;
\item\label{w-spec} $X$ has the weak periodic specification property;
\item\label{a-spec} $X$ has the almost specification property.
\end{enumerate}
Moreover, there exists a synchronized system $X$ (an $S$-gap shift) fulfilling \eqref{mix}--\eqref{rel-p}, but none of \eqref{spec}-\eqref{a-spec}.
\end{thm}
\begin{proof}
It follows from \cite{EKO} and the results mentioned above.
\end{proof}

\bibliographystyle{amsplain}
\bibliography{bibliography}
\end{document}